\documentclass[10pt,reqno]{amsart}
\usepackage{amsmath, latexsym, amsfonts, amssymb,amsthm, amscd,epsfig,enumerate}
\usepackage{color, float, caption, subcaption}

\advance\hoffset -.75cm

\usepackage[usesnames]{xcolor}

\definecolor{refkey}{gray}{.75}
\definecolor{labelkey}{gray}{.75}

\newcommand{\Z}{\mathbb Z}
\newcommand{\N}{\mathbb N}
\newcommand{\Prob}{\mathbb P}
\newcommand{\E}{\mathbb E}

\newcommand{\cB}{\mathcal{B}}

\newcommand{\pr}{\mathbb P}

\newcommand{\ident}{{\mathchoice {\rm 1\mskip-4mu l} {\rm 1\mskip-4mu l}
{\rm 1\mskip-4.5mu l} {\rm 1\mskip-5mu l}}}

\newtheorem{teo}{Theorem}[section]
\newtheorem{lem}[teo]{Lemma}
\newtheorem{cor}[teo]{Corollary}

\newtheorem{pro}[teo]{Proposition}
\newtheorem{defn}[teo]{Definition}

\title
{A stochastic model for the evolution of species with random fitness}

\author[D.~Bertacchi]{Daniela Bertacchi}
\address{D.~Bertacchi, Dipartimento di Matematica e Applicazioni,
Universit\`a di Milano--Bicocca,
via Cozzi 53, 20125 Milano, Italy.}
\email{daniela.bertacchi\@@unimib.it}

\author[J.~Lember]{Juri Lember}
\address{J.~lember, Institute of Mathematics and Statistics,
University of Tartu,
J. Liiv 2, 50409 Tartu, Estonia.}
\email{juri.lember\@@ut.ee}

\author[F.~Zucca]{Fabio Zucca}
\address{F.~Zucca, INdAM, GNAMPA, Piazzale Aldo Moro 5, 
000185 Roma, Italy.}
\email{fabiozucca.math\@@gmail.com}

\begin{document}

\begin{abstract}
We generalize the evolution model introduced by Guiol, Machado and Schinazi (2010).
In our model at odd times a random number $X$ of species is created.
Each species is endowed with a random
fitness with arbitrary distribution on $[0,1]$. 
At even times a random number $Y$ of species is removed, killing the species 
with lower fitness. We show that there is a critical fitness $f_c$
below which the number of species hits zero i.o.~and above of which
this number goes to infinity. We prove uniform convergence for the distribution
of surviving species and describe the phenomena which could not be observed
in previous works with uniformly distributed fitness.
\end{abstract}


\maketitle


\section{Introduction}
\label{sec:intro}

During the history of our planet, species have emerged and have become extinct,
some have lasted a relatively brief period, others are still present in a more 
or less unchanged form after millions of years.
It is widely accepted that the driving engine of evolution is natural
 selection or ``survival of the fittest''. It is therefore interesting
 to provide mathematical models for the evolution of species.
 
Guiol, Machado and Schinazi \cite{cf:Schi} proposed a model where 
creation and deletion of species is driven by chance in the sense that 
at each step
with probability $p$ one new species is created and its fitness is
chosen uniformly in $[0,1]$, while with  probability $1-p$ the least fit species
(if there are species alive at that time) is removed.
One motivation for the study of this model is that its long-term behaviour
is similar to the one which simulations show for the Bak-Sneppen model:
there is a critical value for the fitness and species with smaller fitness
disappear, while species will a larger fitness persist indefinitely.
Bak and Sneppen \cite{cf:baksneppen} modelled a simple ecosystem where the population
size is constant and at each step not only the least fit is removed, but also 
its neighbours are replaced by new species (proximity may be seen as representing
ecological links between species). It has proven difficult to obtain rigorous results
for this model (see for instance \cite{cf:Meester}) and this motivates the search for
similar, more tractable models.

Several papers have studied the GMS model:
\cite{cf:ben-ari} gives a law of the iterated logarithm
and a central limit theorem for number of species with supercritical fitness
which go extinct (this number is negligible with respect to $n$); 
\cite{cf:Mach} studies the maximal fitness ever appeared in the subcritical case.

The model has been generalized in \cite{cf:Volk} and \cite{cf:ben-ari2}:
there is still a toss of a coin to decide for creation or deletion, but 
instead of adding/removing one species at a time, increments 
 are arbitrary random variables.
 Even with these assumptions, the same cut-off phenomenon of \cite{cf:Schi}
appears.

In the original GMS, the lengths of subsequent births and deaths are geometrically distributed random variables (with parameters which sum up
to 1)
and in \cite{cf:ben-ari2, cf:Volk} they are geometrical convolutions of certain laws (where the parameters of these
geometrically distributed number of convolutions, again, sum up to 1). In our model we group all subsequent creations and deletions:
the length of subsequent creations
$\{X_n\}_{n \in \N}$ and the length of subsequent annihilations 
$\{Y_n\}_{n \in \N}$ are such that $\{(X_n, Y_n)\}_{n \in \N}$ is an i.i.d.~sequence with arbitrary
distribution. Whence our results apply to the models in \cite{cf:ben-ari2,cf:Schi,cf:Volk} (see Section~\ref{subsec:previous}).  
Besides, in the older papers the fitness is assigned uniformly while we use a general distribution $\mu$. 
If $\mu$ has atoms, a new phenomenon appears: there might be a fitness which acts as a barrier eventually protecting all
species with higher fitness (see Corollary~\ref{rem:dynamics} and the subsequent discussion for details).

Here is the outline of the paper. In Section~\ref{sec:basic} we give the formal construction
of the process and the necessary definitions. We state our main result, Theorem~\ref{th:genshape1}, which
describe the asymptotic expression of the proportion of species in a generic (Borel) range of fitness.
The asymptotic behavior of a single fitness is described by Theorem~\ref{th:gen1}. Corollary~\ref{rem:dynamics}
and the subsequent discussion gives some details on the number of species which are killed. Section~\ref{subsec:previous}
is devoted to a detailed comparison with previous works; we explain why our work is a generalization of the previous
models and which new phenomena arise. 

In Section~\ref{sec:examples} we study some examples: the original GMS with an atomic measure $\mu$ (see Section~\ref{subsec:original}),
a Markov model which cannot be treated by using 
previously known results (see Section~\ref{subsec:markov}) and  a model which is related to Branching Processes
(see Section~\ref{subsec:warm}). We also give a counterexample to be compared with Theorem~\ref{th:genshape1}(2).

All the proofs are in Section~\ref{sec:proofs} which contains a couple of results which are
worth mentioning: a \textit{Law of Large Numbers} (Proposition~\ref{pro:slln}) and Proposition~\ref{pro:typesofrecurrence}
which identifies the set of fitness which become empty i.o.~(and the total amount of time they are empty).

\section{The process and its asymptotic behaviour}
\label{sec:basic}

We start by giving a formal description of the process.

Let $\{X_n, Y_n, f_{n,i}\}_{n,i \in \N}$ be a family of nonnegative random variables and, for all $n \in \N$, 
denote by $f_n$ the sequence $\{f_{n,i}\}_{i \in \N}$. Suppose that 
\begin{enumerate}
 \item for every $n \in \N$, $(X_n, Y_n)$ and $f_{n}$ are independent, 
 \item $\{(X_n,Y_n, f_n)\}_{n \in \N}$ are i.i.d.
 \item all $f_{n,i}$ are distributed according to a measure $\mu$ on $\mathbb{R}$.
 \end{enumerate}
Roughly speaking, $X_n$ counts the new species at time $n$, $Y_n$ counts the deaths 
 and $f_{n,i}$ the fitness of a newly created species.
In order to avoid trivial cases we suppose that $\E[X_k]$ and $\E[Y_k]$ are both in $(0,+\infty]$; moreover 
we assume that at least one of these two expected values is finite.
Note that in this case $\{X_n\}_{n \in \N}$, $\{Y_n\}_{n \in \N}$ and $\{X_n-Y_n\}_{n \in \N}$ are all i.i.d.~families, but
$X_i$ and $Y_i$ might be dependent. 
From now on, we will denote by $(X,Y)$ a couple with the same law as $(X_1,Y_1)$.
For every fixed $n \in \N$ also $\{f_{n,i}\}_{i \in \N}$ might be dependent (for instance
they can be generated by a Markov Chain or $f_{n,1}=f_{n,i}$ for all $i \in \N$).

We will assume that $\mu([0,1])=1$;
there is no loss of generality, since any measure on $\mathbb{R}$ can be mapped to a measure supported in $[0,1]$.
We denote its cumulative distribution function by $F=F_\mu$ and we define $F(f^-):=\lim_{a \to f^-} F(a)$. 

Let $Z_n$ be the number of species alive at time $n$. We start at time $0$ with 
$Z_0=0$ ($Z_0$ could be a random variable with an arbitrary distribution on $\N$).

At time $1$, $X_1$ species are generated and to each of them we assign a random fitness with law $\mu$.
More precisely the fitness of the $i$-th created species is $f_{1,i}$ for all $1\le i\le X_1$.
Thus $Z_1=Z_0+X_1$. The procedure is repeated at any odd time: $Z_{2k+1}=Z_{2k}+X_{k+1}$, meaning that
$X_{k+1}$ species are created and their fitness $f_{k+1,1},\ldots,f_{k+1,X_{k+1}}$ are assigned.
For any set $A \subseteq [0,1]$ we denote by $Z_n(A)$ the total number of species
alive at time $n$ and with fitness in $A$. The fitness of a species does not change during its entire lifetime,
and species may disappear only at even times.

At time $2k+2$, a number $Y_{k+1}\wedge Z_{2k+1}$ of species are removed and removal starts from the least fit.
This means that $Z_{2k+2}=0\vee (Z_{2k+1}-Y_{k+1})$. Thus if $Y_{k+1}\ge Z_{2k+1}$ then $Z_{2k+2}(A)=0$
for all $A \subseteq [0,1]$. Otherwise,
let $x_-:=\max\{x \in [0,1]\colon Z_{2k+1}([0,x]) \le Y_{k+1} \}$ and  $x_+:=\min\{x \in [0,1]\colon Z_{2k+1}([0,x]) \ge Y_{k+1} \}$.
All species with fitness not larger than $x_-$ are removed and $Z_{2k+2}(A)=0$ for all $A\subseteq[0,x_+)$.
A number $M_{k+1}:=Y_{k+1}-Z_{2k+1}([0,x_-])$ of species
is removed from the set of species with fitness equal to $x_+$: $Z_{2k+2}(\{x_+\})=Z_{2k+1}(\{x_+\})-M_{k+1}$
and  $Z_{2k+2}(A)=Z_{2k+1}(A)$ for all $A\subseteq(x_+,1]$.

%


Given a Borel set $A \subseteq [0,1]$
such that $\mu(A)>0$,
we define the number of species created in $A$ as
\begin{equation}\label{eq:XYtilde}
\widetilde X_n= \sum_{i=1}^{X_n} \ident_{A}(f_{n,i}).
\end{equation}
 By our assumptions, for any $A$, we have
that  $\{(\widetilde X_n,Y_n)\}_{n \in \N}$ are i.i.d.~and $\E[\widetilde
X _n]=\mu(A)\E[X_n]$ (where $a\cdot(+\infty)= +\infty$, if $a > 0$ and
$0\cdot\infty=0$). 
Henceforth, an interval $I \subseteq [0,1]$ (either closed or not) such that $0 \in I$ is called a \textit{left interval}. 
We note that, for a left interval $I$ such that $\mu(I)>0$, 
$\{Z_{2n}(I)\}_{n \in \N}$ is the 
queuing process (see \cite[Chapter VI.9]{cf:Feller2})
associated to the i.i.d.~increments $\{\widetilde X_n-Y_n\}_{n \in \N}$ (see Section~\ref{sec:proofs} for details).

We will often make use of the expected value $\E[\alpha X-Y]$ where $\alpha \in [0,1]$.
If $\E[X]=+\infty>\E[Y]$ and $\alpha>0$ then $\E[\alpha X-Y]:=+\infty$;  
if $\E[Y]=+\infty>\E[X]$ then $\E[\alpha X-Y]:=-\infty$ for all $\alpha \in [0,1]$. 

We define the critical parameter:
\begin{equation}\label{eq:criticalf}
  \begin{split}
f_c&:= \inf\{f \in \mathbb{R} \colon F(f)> \E[Y]/\E[X]\} 
 \end{split}
\end{equation}
Note that 
when $\E[Y] \ge \E[X]$ then $f_c=+\infty$, otherwise
 $f_c$ is the only solution of
$F(f_c) \ge \E[Y]/\E[X] \ge F(f_c^-)$, where both inequalities turn into equalities if and only if 
$\mu(\{f_c\})=0$.
 
When $\E[Y] < \E[X]<+\infty$,
we define the following probability measure (on Borel sets $A \subseteq [0,1]$) and its cumulative distribution function
  \begin{equation}\label{eq:cdlshape}
\begin{split}
\pr_\infty(A)&:=
\frac{\mu(A\cap (f_c,1])\E[X]+ \ident_A(f_c) \E[\mu([0,f_c])X-Y]}{\E[X-Y]},\\
 F_\infty(f)&:=
 \begin{cases}
0 & f < f_c,\\
\frac{\E[F(f)X-Y]}{\E[X-Y]} & f \ge f_c.\\
   \end{cases}\\
   \end{split}
  \end{equation}

\begin{defn}\label{def:extinction}
 Let $A \subseteq [0,1]$. We say that 
 \begin{enumerate}[(i)]
  \item there is extinction in $A$ if and only if  $Z_n(A)=0$ infinitely often a.s.;
  \item there is survival in $A$ if and only if for all $n\in\N$ such that $\Prob(Z_n(A)>0)>0$ we have 
  $\Prob(Z_m(A)>0, \forall m\ge n|Z_n(A)>0)>0$. 
 \end{enumerate}
 When $A=\{f\}$ we speak of extinction and survival of the fitness $f$. 
%
\end{defn}
It is a consequence of the following theorem that, when $A \subseteq [0,1]$ is a Borel set, either there is
extinction in $A$ or there is survival. Indeed, if there is no extinction in $A$ then $\pr_\infty(A) >0$, thus
$Z_n(A) \to +\infty$ almost surely. By a standard argument this implies survival.

%
%
%
  
  \begin{teo}[\textbf{Shape Theorem}]\label{th:genshape1}
  $\ $ \\
  \begin{enumerate}
   \item 
 For all sets $A \subseteq [0,1]$ such that $\mu(A \setminus [0,f_c))=0$,
   there is extinction in $A$ and
   $Z_{n}(A)/n \stackrel{n \to +\infty}{\longrightarrow} 0$ uniformly with respect to $A$ 
   almost surely. If $F(f_c)=\E[Y]/\E[X]$ then the same holds for all $A \subseteq [0,1]$ such that $\mu(A \setminus [0,f_c])=0$.
 \item  If, for every $n$, $\{f_{n,i}\}_{i \in \N}$ are i.i.d and $\E[X]=+\infty>\E[Y]$ then  we have 
 that $Z_n/n\stackrel{n \to +\infty}{\longrightarrow}\infty$ and $Z_{n}(A)/Z_{n} \stackrel{n \to +\infty}{\longrightarrow} \mu(A)$ 
 a.s.~(for Borel sets  $A$ such that $\mu(A)>0$).
   \item
   If $\E[X-Y] \in (0, +\infty)$ then $Z_{n}/n \stackrel{n \to +\infty}{\longrightarrow} \E[X-Y]/2$ a.s.~and    
  \begin{equation}\label{eq:shape1}
  \pr \Big ( \frac{Z_{n}(A)}{Z_n} \stackrel{n \to +\infty}{\longrightarrow} \pr_\infty(A), \, \textrm{for every Borel set }A \subseteq [0,1]
  \Big )=1.
   \end{equation}
 Moreover 
 \begin{equation*}
 \sup_{f \in [0,1]} \Big|\frac{Z_n([0,f])}{Z_n}-F_\infty(f)\Big|\to
0, \qquad \textrm{as }n \to +\infty, \ a.s.
\end{equation*}
 \end{enumerate}
  \end{teo}
%
%
  It is worth noting that, as a consequence of Theorem~\ref{th:genshape1}(1), whenever $\E[\mu(I)X-Y] \in [-\infty, 0]$ 
    for some left interval $I$, then $Z_{2n}(I)=0$ infinitely often
  a.s.; nevertheless $Z_{2n}(I)$ has a non-trivial limit in law (see Proposition~\ref{pro:typesofrecurrence}(3) for details).
  This implies that when $\E[X]<+\infty$ and $\E[X] \le \E[Y] \le +\infty$ 
then all fitness go extinct.
  
  The example given in Section~\ref{subsec:counterexample} shows that, if $\{f_{n,i}\}_{i \in \N}$ are just dependent, 
  then the conclusion in Theorem~\ref{th:genshape1}(2) might be false.
  
  
%
%

The following theorem describes the long-term behaviour of a fixed fitness.
Note that all $f>f_c$ belong to case (1), while all $f<f_c$ 
belong to (2). If $f=f_c$, then case (2) applies if
and only if $F(f_c)=\E[Y]/\E[X]$.

\begin{teo}[\textbf{Extinction and survival}]\label{th:gen1}
Let $f \in [0,1]$.
 \begin{enumerate}
  \item If $\E[F(f) X-Y] \in (0,+\infty]$
  then there is survival in $[0,f]$ and
  the fitness $f$ survives. Moreover, $\lim_{n \to \infty}Z_n([0,f])=\infty$ a.s.~and, if, 
  $\mu(\{f\})>0$ then $\lim_{n \to \infty} Z_n(\{f\}) =\infty$
  almost surely.
  \item If $\E[F(f) X-Y]\in [-\infty,0]$
  then there is extinction in $[0,f]$.
\end{enumerate}
\end{teo}

Denote by $K_n(A)$ the number of species killed in $A$ up to time $n$ and by $\tau_n(A)$ the total number of epochs that there are
no species in $A$ up to time $n$.
From Theorem~\ref{th:genshape1}, if $\E[X-Y]>0$ then, as $n \to +\infty$,
\begin{equation}\label{eq:killing}
  \frac{K_n(A)}{n} = \frac{\sum_{i=1}^{\lfloor n/2 \rfloor} \widetilde X_i -Z_{n}(A)}{n} \sim
 \frac12 \mu(A \cap [0,f_c]) \E[X]-\frac12 \ident_A(f_c) \E[\mu([0,f_c])X-Y] \qquad a.s.
\end{equation}
where $\widetilde X_n$ is the number of species created in the Borel set $A$ (see equation~\eqref{eq:XYtilde}).

\begin{cor}\label{rem:dynamics}
 If $\E[X-Y] >0$ then
 \begin{enumerate}
  \item  $\displaystyle \lim_{n \to +\infty} K_n((f_c,1])/n =0$ a.s.;
\item  $K_n([f_c,1]) \to +\infty$ a.s.;
\item  If $F(f_c)> \E[Y]/\E[X]$ then $\displaystyle \sup_{n \in \N} K_n((f_c,1]) <+\infty$ a.s., otherwise $\displaystyle \lim_{n \to +\infty} K_n((f_c,1]) = +\infty$ a.s.;
 \item  If $F(f_c^-)< \E[Y]/\E[X]$ then $\displaystyle \lim_{n \to +\infty} K_n([f_c,1])/n >0$ a.s.,
 otherwise $\displaystyle \lim_{n \to +\infty} K_n([f_c,1])/n =0$ a.s.;
 \item  If $f>f_c$ then $\displaystyle \sup_{n \in \N} K_n([f,1]) <+\infty$ almost surely.
\end{enumerate}
 \end{cor}

Here is a more explicit description.
First of all, by (5) a.s.~there are no more species killed in $[f,1]$ eventually as $n \to +\infty$ but by (2)
the number of species killed in $[f_c,1]$ diverges almost surely.

If $\mu(\{f_c\})=0$ then
$F(f_c)=\E[Y]/\E[X]=F(f_c^-)$, so that by (4) $K_{n}([f_c,1])/n$  goes to
zero  a.s.~and $\tau_n([0,f_c])/n \to 0$ almost surely as $n \to +\infty$
(see Proposition~\ref{pro:typesofrecurrence}(2)).

If $\mu(\{f_c\})>0$ then we have the following possibilities:
\begin{itemize}
 \item $F(f_c)>\E[Y]/\E[X]=F(f_c^-)$ then by (4) $K_{n}([f_c,1])/n$ goes to
zero  almost surely. Moreover, 
by Theorem~\ref{th:gen1}(1) we have $Z_n(\{f_c\}) \to +\infty$ almost surely as $n \to +\infty$,
 implying that the species killed in $[f_c, 1]$ eventually will have fitness $f_c$
 almost surely. Even though the number of species of fitness $f_c$ which are killed
 diverges, by equation~\eqref{eq:shape} the  fraction of species alive with fitness
 $f_c$ converges to $(F(f_c)-\E[Y]/\E[X]) \cdot \E[X]/2>0$. 
  Also,
 $\tau_n([0,f_c))/n$ goes to $0$ (see again Proposition~\ref{pro:typesofrecurrence}(2)).
 \item $F(f_c)>\E[Y]/\E[X]>F(f_c^-)$ then, just as before, a.s.~the species killed in $[f_c, 1]$ eventually will have fitness $f_c$
 and the  fraction of species alive with fitness
 $f_c$ converges to the same positive limit.
 This time $K_n(\{f_c\})/n$ has a positive limit: $-\E[F(f_c^-)X-Y]/2$ and
 $\tau_n([0,f_c))/n$ converges to a positive limit almost surely as $n \to +\infty$
(see Proposition~\ref{pro:typesofrecurrence}(3)). 
  
 \item $F(f_c)=\E[Y]/\E[X]>F(f_c^-)$ then, by Theorem~\ref{th:genshape1}(2),
 every species with fitness $f_c$
 is eventually  killed a.s.~and  $K_n([f_c,1])/n$ converges to $-\E[F(f_c^-)X -Y]/2>0$.
  But $K_n((f_c,1])/n$ 
tends to 0, a.s., thus outside a negligible proportion, the killed
species all have fitness
 $f_c$, whence $K_n(\{f_c\})/n$ has the same positive limit as before.  
Finally,
$\tau_n([0,f_c])/n \to 0$ almost surely as $n \to +\infty$
(see Proposition~\ref{pro:typesofrecurrence}(2)).
\end{itemize}

\subsection{Comparison with previous works}
\label{subsec:previous}
Our process extends those appeared in \cite{cf:ben-ari2,cf:Schi,cf:Volk}. 
Aside from our general choice for the fitness law, the birth-and-death mechanism
that we study is more general than those adopted in these papers.


One way to see the original GMS (see \cite{cf:Schi}) as 
a particular case of our process is by
observing that the random sequences of consecutive births $X_k$ and consecutive
deaths $Y_k$
have right-shifted Geometric distribution with parameter $1-p$ and $p$ respectively.

In general, consider a process $\{\overline{Z}_n\}_{n
\in \N}$ where at each step either a species is created (along with
its fitness) or the least-fit species, if any, is removed.  Denote
by $X_1>0$ the length of the first stretch of ``creations'',
followed by a stretch of ``annihilations'' of length $Y_1>0$, then
another stretch of ``creations'' of length $X_2$ followed by a
stretch of $Y_2$ ``annihilations'' and so on. Suppose that
$\{X_n\}_{n \in \N}$ and $\{Y_n\}_{n \in \N}$ are two
i.i.d.~sequences. It is clear that there is a connection between our process
and this one, namely for every set $A$,
$Z_{n}(A)=\overline{Z}_{N_n}(A)$ where $N_n=\sum_{i=1}^{{\lfloor}(n+1)/2{\rfloor}}X_i
+
 \sum_{i=1}^{{\lfloor}n/2{\rfloor}}Y_i$.

In particular if $n$ is even and $k\in (N_n,
N_{n+1})$, then  $\overline{Z}_{k}(A)$ is nondecreasing, while if
$k\in (N_{n+1}, N_{n+2})$, then  $\overline{Z}_{k}(A)$ is
nonincreasing.  Proposition~\ref{pro:slln} shows that for
every left interval $I$
$$ \frac{Z_{n}(I)}{n}=\frac{\overline{Z}_{N_n}(I)}{n}\to \frac{\E[\mu(I)X-Y]}{ 2},\quad   {\rm a.s.}$$
When $\E[X+Y]<\infty$, then the monotonicity of
$\overline{Z}_{k}$ between $N_{n}$ and $N_{n+1}$, implies
\[
\frac{\overline{Z}_n(I)}{n}\to \frac{\E[\mu(I)X-Y]}{2}\frac{2}{\E[X+Y]}=
\frac{\E[\mu(I)X-Y]}{\E[X+Y]},\quad \rm{a.s.}
\]
Therefore, the long-term behaviour of $\{\overline{Z}_n(I)\}_{n \in
\N}$ can be derived simply by studying $\{Z_n(I)\}_{n \in \N}$.

Our work can also be considered as a generalization of \cite{cf:ben-ari2} and \cite{cf:Volk}
whose models are essentially equivalent.
Indeed,
in \cite{cf:ben-ari2}, a single family of $\Z$-valued
variables $\{U_n\}_{n\in\N}$ is considered. In this process,
$U_n>0$ means that $U_n$ species are created, while $U_n<0$ means
that $-U_n$ species are killed.  In this case the laws of length of
a ``creation'' stretch $X_i$ and ``annihilation'' stretch $Y_i$ are
necessarily geometric random convolutions of the law of $U_n$
conditioned on $\{U_n>0\}$ and $\{U_n<0\}$ respectively. Moreover,
the sum of the parameters of these geometric convolutions must be
$1-\pr(U_1=0)$. {Therefore,} a model constructed
from the variables $\{U_n\}_{n\in\N}$ can be considered as a
particular case of our model: take for instance $X_n:=U_n \ident_{\{U_n >0\}}$,
$Y_n:=-U_n \ident_{\{U_n <0\}}$ and consider the process $\{Z_{2n}(A)\}_{n \in \N}$.
In Section~\ref{subsec:markov} we consider a particular case of
our process which cannot be obtained with a single family of
variables describing simultaneously creations and annihilations.

Observe that in Theorem~\ref{th:genshape1} we used $Z_n$ as a normalizing factor for $Z_n(A)$ 
but there are two other
natural choices: $n$ (to compare with \cite{cf:ben-ari2, cf:Volk}) and $N_n$ (to compare with \cite{cf:ben-ari, cf:Schi}).
\\
If $\E[X_i+Y_i]<+\infty$ then, by the Strong Law of Large Numbers (SLLN), $N_n \sim n \E[X+Y]/2$ almost surely as $n \to +\infty$.
If, in addition, $\E[X-Y] \in (0, +\infty)$ then by Proposition~\ref{pro:slln}  we have
\[
 Z_n\sim n \E[X-Y]/2 \sim N_n \frac{\E[X-Y]}{\E[X+Y]} \qquad a.s.
\]
as $n \to +\infty$.
Hence Theorem~\ref{th:genshape1}(3) can be equivalently written in terms of the timescale $n$ or $N_n$ 
(in this last case we obtain a generalization of Proposition~\ref{pro:slln}(1) to Borel sets).
 \\
 If $\E[X]=+\infty>\E[Y]$ then  $Z_n\sim N_n$ almost surely 
 as $n \to +\infty$. Indeed one can use the same kind of arguments used in the proof of 
 Theorem~\ref{th:genshape1}(2), to prove that $Z_n$ and $\sum_{i=1}^{{\lfloor}(n+1)/2{\rfloor}}X_i$ are 
 asymptotic and the remaining terms are negligible.
Roughly speaking, changing timescale turns out to be just a linear rescaling.

We note that for the GMS model and its generalizations, with $\mu\sim \mathcal{U}([0,1])$ (where $\mathcal{U}(I)$ is
the uniform distribution on $I$), the fraction
of surviving species in any $I\subseteq[f_c,1]$ is proportional to $\mu(I)$. This is still true in our case
when $I\subseteq(f_c,1]$, but it does not hold for instance if $I=[f_c,b]$ and
$F(f_c^-)< \E[Y]/\E[X]$. Moreover if $\mu\sim \mathcal{U}([0,1])$ then $K_n([f_c,1])/n\to 0$ 
(the exact rate of convergence for the GMS is studied in \cite{cf:ben-ari}), while again
this needs not to be true if $f_c$ is an atom for $\mu$.

\section{Examples and counterexamples}\label{sec:examples}

\subsection{The original GMS model}\label{subsec:original}

The original GMS process can be seen as the particular case where $X$ has a geometric law $\mathcal{G}(1-p)$ 
while $Y$ has a geometric law $\mathcal{G}(p)$.
Thus 
$\E[X_i]=1/(1-p)$, $\E[Y_i]=1/p$ and $\E[Y_i]/\E[X_i]=(1-p)/p$ is the relevant term for computing $f_c$ according to 
equation~\eqref{eq:criticalf}.

In this example we choose
$\mu:=\alpha \delta_{1/2} + (1-\alpha)\nu$ (where $\nu \sim  \mathcal{U}([0,1])$); the case $\alpha=0$ is discussed in \cite{cf:Schi}. 
Roughly speaking, every time a new species is born
we toss a (possibly biased) coin: with probability $\alpha$ we assign to the new species a fitness $1/2$ and with probability $1-\alpha$
the fitness is drawn uniformly and independently in $[0,1]$. 
Clearly
\begin{equation}\label{eq:Fexample}
 F(f)=
 \begin{cases}
(1-\alpha)f & f \in [0,1/2) \\
(1-\alpha)f+\alpha & f \in [1/2,1].
 \end{cases}
\end{equation}
For every Borel set $A \subseteq [0,1]$, $\E[\mu(A)X-Y]=\mu(A)/(1-p)-1/p$ and, according to equation~\eqref{eq:shape1},
\[
 \lim_{n \to +\infty} \frac{Z_{n}(A)}{Z_n}= \mu(A \cap (f_c,1])\frac p{2p-1} +\ident_A(f_c) \frac{\mu([0,f_c])p-(1-p)}{2p-1},
\]
where $f_c$ is given by equation~\eqref{eq:criticalf} and it is the unique solution in $[0,1]$ of $F(f_c) \ge \E[Y]/\E[X]=(1-p)/p  \ge F(f_c^-)$.
More interesting  is the cumulative limit distribution (see equation~\eqref{eq:cdlshape})
  \[
   \lim_{n \to +\infty} \frac{Z_n([0,f])}{Z_n}=
   \begin{cases}
0 & f \in [0,f_c)\\
F(f) \frac{p}{2p-1}-\frac{1-p}{2p-1} & f \in [f_c,1].\\
\end{cases}
   \]

To avoid useless complications, 
we discuss just the ``fair coin'' case $\alpha=1/2$. 
In this case we have
\[
 f_c=
 \begin{cases}
  2(1-p)/p & p \in (4/5,1]\\
  1/2 & p \in [4/7,4/5]\\
  (2-3p)/p & p \in (1/2, 4/7).
 \end{cases}
\]
There are five typical situations that we can explore and they are represented by the following table where we choose
$(1-p)/p=7/8, \, 3/4, \, 1/2,\, 1/4, \, 1/8$.
\medskip

\begin{center}
\resizebox{\textwidth}{!}{
\begin{tabular}{|c|c|c|c|c|c|c|c|}
 \hline 
 $p$ & $f_c$ & $F(f_c)$ & $F(f_c^-)$ & ${\E[F(f_c)X-Y]}$ & ${\E[F(f_c^-)X-Y]}$ & \multicolumn{2}{c|}{$\lim_{n \to +\infty} Z_n([0,f])/Z_n$}\\
  & & & & & & \multicolumn{2}{c|}{\qquad c.d.f.  \qquad  \qquad \qquad \qquad \qquad Law } \\
 \hline
 $8/15$ & $3/4$ & $7/8$ & $7/8$ & $0$ & $0$ & \phantom{\Big (} $(4f-3)\ident_{[3/4,1]}(f)$ \phantom{\Big (}& $\mathcal{U}([3/4,1])$\\
 \hline
 $4/7$ & $1/2$ & $3/4$ & $1/4$ & $0$ & $-7/6$ & \phantom{\Big (}$ (2f-1)\ident_{[1/2,1]}(f) $ \phantom{\Big (} & $\mathcal{U}([1/2,1])$ \\
 \hline
 $2/3$ & $1/2$ & $3/4$ & $1/4$ & $3/4$ & $-3/4$ & \phantom{\Big (} $f \ident_{[1/2,1]}(f) $ \phantom{\Big (} & $\frac12\delta_{1/2}+\frac12 \mathcal{U}([1/2,1])$ \\
 \hline 
 $4/5$ & $1/2$ & $3/4$ & $1/4$ & $5/2$ & $0$ & \phantom{\Big (} $\frac{2f+1}{3}\ident_{[1/2,1]} $\phantom{\Big (} & $\frac23 \delta_{1/2}+\frac13\mathcal{U}([1/2,1])$ \\
 \hline
 $8/9$ & $1/4$ & $1/8$ & $1/8$ & $0$ & $0$ & \phantom{\Big (} $\frac{4f-1}{7}\ident_{[1/4,1]}(f) + \frac47\ident_{[1/2,1]}(f)$\phantom{\Big (}  & $\frac47 \delta_{1/2} +\frac37 \mathcal{U}([1/4,1]) $\\
 \hline
\end{tabular}
}
\end{center}
\medskip

\subsection{The Markov case}\label{subsec:markov}

Let the birth-death process be now a Markov chain with transition matrix
\begin{equation}
\left(
    \begin{array}{cc}
      p & 1-p \\
      1-q & q \\
    \end{array}
  \right)\end{equation}
starting from a birth.
Thus the probability of a birth after the birth $P_{++}=p$, the
probability of death after the birth is $P_{+-}=1-p$ and so on. 
This can be seen as a particular case of our process 
where $X$ has a geometric law $\mathcal{G}(1-p)$ 
while $Y$ has a geometric law $\mathcal{G}(1-q)$.
We assume that $p>q$; clearly $\E[Y]/\E[X]=(1-p)/(1-q)$. 

As before we choose
$\mu:=\alpha \delta_{1/2} + (1-\alpha)\nu$ (where $\nu \sim  \mathcal{U}([0,1])$); thus the cumulative distribution function is 
still given by equation~\eqref{eq:Fexample}.

Now $\E[\mu(I)X-Y]=\mu(I)/(1-p)-1/(1-q)$ and, according to equation~\eqref{eq:cdlshape},
  \[
   \lim_{n \to +\infty} \frac{Z_n([0,f])}{Z_n}=
   \begin{cases}
0 & f \in [0,f_c)\\
F(f) \frac{1-q}{p-q}-\frac{1-p}{p-q} & f \in [f_c,1]\\
\end{cases}
   \]
where $f_c$ is the unique solution in $[0,1]$ of $F(f_c) \ge (1-p)/(1-q)  \ge F(f_c^-)$.

As before, 
we discuss just the ``fair coin'' case $\alpha=1/2$. 
In this case we have
\[
 f_c=
 \begin{cases}
  2(1-p)/(1-q) & (q+3)/4 <p \le 1\\
  1/2 &  (1+3q)/4 \le p \le (3+q)/4\\
 (1+q-2p)/(1-q) & q < p < (1+3q)/4.
 \end{cases}
\]
We retrieve the same typical cases as before by choosing
$(1-p)/(1-q)=7/8, \, 3/4, \, 1/2,\, 1/4, \, 1/8$.

\begin{center}
\resizebox{\textwidth}{!}{
\begin{tabular}{|c|c|c|c|c|c|c|c|}
 \hline 
 $(p,q)$ & $f_c$ & $F(f_c)$ & $F(f_c^-)$ & ${\E[F(f_c)X-Y]}$ & ${\E[F(f_c^-)X-Y]}$ & \multicolumn{2}{c|}{$\lim_{n \to +\infty} Z_n([0,f])/Z_n$}\\
  & & & & & & \multicolumn{2}{c|}{\qquad c.d.f.  \qquad  \qquad \qquad \qquad \qquad Law } \\
 \hline
 $(2/9,1/9)$ & $3/4$ & $7/8$ & $7/8$ & $0$ & $0$ & \phantom{\Big (} $(4f-3)\ident_{[3/4,1]}(f)$ \phantom{\Big (}& $\mathcal{U}([3/4,1])$\\
 \hline
 $(2/5,1/5)$ & $1/2$ & $3/4$ & $1/4$ & $0$ & $-5/6$ & \phantom{\Big (}$ (2f-1)\ident_{[1/2,1]}(f) $ \phantom{\Big (} & $\mathcal{U}([1/2,1])$ \\
 \hline
 $(3/4,1/2)$ & $1/2$ & $3/4$ & $1/4$ & $1$ & $-1$ & \phantom{\Big (} $f \ident_{[1/2,1]}(f) $ \phantom{\Big (} & $\frac12\delta_{1/2}+\frac12 \mathcal{U}([1/2,1])$ \\
 \hline 
 $(5/6,1/3)$ & $1/2$ & $3/4$ & $1/4$ & $3$ & $0$ & \phantom{\Big (} $\frac{2f+1}{3}\ident_{[1/2,1]} $\phantom{\Big (} & $\frac23 \delta_{1/2}+\frac13\mathcal{U}([1/2,1])$ \\
 \hline
 $(9/10,1/5)$ & $1/4$ & $1/8$ & $1/8$ & $0$ & $0$ & \phantom{\Big (} $\frac{4f-1}{7}\ident_{[1/4,1]}(f) + \frac47\ident_{[1/2,1]}(f)$\phantom{\Big (}  & $\frac47 \delta_{1/2} +\frac37 \mathcal{U}([1/4,1]) $\\
 \hline
\end{tabular}
}
\end{center}
\medskip

 \subsection{The Branching Process case}
\label{subsec:warm}

In this example we consider the case where $\pr(Y_n=1)=1$ while $X_k$ has a generic discrete distribution on $\N$. In order to avoid a trivial situation
we assume that $\pr(X_k=1)<1$.
In the following, we make use of the generating function of the variables $\{X_n\}_{n \in \N}$, that is, $\Phi(z):= \sum_{n =0}^\infty \pr(X_0=n) z^n$
for all $z \in [0,1]$.
Similarly, the generating function of the random number of species
whose fitness belongs to $[0,f]$ (resp.~$[0,f)$), that is $\widetilde X_n$, is $\Psi_f(z)=\Phi(zF(f)+1-F(f))$ (resp.~$\Psi_{f^-}(z)=\Phi(zF(f^-)+1-F(f^-))$; 
see Lemma~\ref{lem:genfunct}(3) for details.
In this case, clearly, 
\[
 \begin{split}
f_c&:= \inf\{f \in \mathbb{R} \colon F(f) \Phi^\prime(1) > 1\} \\
 \end{split}
\]

The peculiarity of this case is the fact that the process can studied by means of a branching process and the
probability of survival of a fitness can be computed in terms of the probability of survival of the branching process. 

\begin{pro}\label{pro:BPcase}
Let $f \in [0,1]$ and consider the process $\{Z_{n}([0,f])\}_{n \in \N}$.
Denote by $\bar q_f$ the smallest fixed point in $[0,1]$ of the generating function $\Psi_f=\Phi(zF(f)+1-F(f))$.
Then
\begin{equation}\label{eq:BP1}
 \pr(\exists k \ge n \colon Z_{j+2k}([0,f])=0 | Z_{2n+1}([0,f])=i)=
 \begin{cases}
\bar q_f^i  & \text{if }j=1\\
\bar q_f^{i-1} & \text{if }j=2\\
 \end{cases}
\end{equation}
for $i \ge 1$; moreover
$\bar q_f =1$ if and only if $F(f) \le 1/\Phi^\prime(1)$.
In particular, if $f < f_c$, then there is extinction in $[0,f]$
and if $f > f_c$, then there is survival in $[0,f]$. If $f=f_c$,
then there is  extinction in $[0,f_c]$ if and only if $F(f_c) =
1/\Phi^\prime(1)$.

The same holds for the process $\{Z_{n}([0,f))\}_{n \in \N}$ by using $[0,f)$ and $\Psi_{f^-}=\Phi(zF(f^-)+1-F(f^-))$ instead
of $[0,f]$ and $\Psi_f$ respectively. In particular there is extinction in $[0, f_c)$
\end{pro}

\subsection{Counterexample of Theorem~\ref{th:genshape1}(2) for dependent $\{f_{n,i}\}_{i \in \N}$
}\label{subsec:counterexample}

 We define $Y_n:=1$ a.s.~and $f_{n,i}:=f_{n,1}$ for all $i>1$ and $n \in \N$ (where $\{f_{n,1}\}_{n \in \N}$ is 
an i.i.d.~sequence distributed according to $\mu$).
We construct the sequence $\{X_n\}_{n \in \N}$ as $X_n:=g(H_n)$, for a suitable choice of an i.i.d.~sequence $\{H_n\}_{n \in \N}$ 
and a function $g$.

Let $\{H_n\}_{n \in \N}$ be an i.i.d.~sequence 
such that 
$\pr(H_n=i):=1/2^{i}$ for all $i \in \N\setminus\{0\}$. We define $n_i:=i(i+1)/2$ for all $i \in \N$, hence
$\pr(H_n \le n_{i+1}|H_n > n_i)=1-1/2^{i+1}$ for all $i \in \N$.

Let us define  
$T_{k}:=\min\{i \colon H_i > n_k\}$ for all $k \in \N$ (clearly $T_0=1$);
note that $T_k \sim \mathcal{G}(1/2^{n_k})$. Since
$\lim_{m \to +\infty} \pr(T_k \le m, H_{T_k} \le n_{k+1})=\pr(H_{T_k} \le n_{k+1})=1-1/2^{k+1}$ then
there exists $\tau_k \in \N$ such that $\pr(T_k \le \tau_k, H_{T_k} \le n_{k+1}) \ge 1-1/2^k$.
The sequence $\{\tau_k\}_{k \ge 1}$ can be always constructed iteratively as a nondecreasing sequence.
It is not difficult to prove, by using the Borel-Cantelli Lemma, that the event
$\Omega_0 := \bigcap_{k \in \N} \{H_{T_k} \le n_{k+1}, T_k \le \tau_k\}$ has positive probability.

We are now ready to define $g(i):= (k+1)! \prod_{j=0}^{k} \tau_j$ for all $i=n_{k}+1, \ldots, n_{k+1}$ (for all $k \in \N$);
$X_n:=g(H_n)$ for all $n \in \N$.
On $\Omega_0$ we have
\[
 \frac{X_{T_k}}{\sum_{i=1}^{T_k} X_i} \ge \frac{X_{T_k}}{(T_k-1)g(n_{k})+ X_{T_k}} \ge
 \frac{(k+1)! \prod_{j=0}^{k} \tau_j}{\tau_{k}\, k! \prod_{j=0}^{k-1} \tau_j+(k+1)! \prod_{j=0}^{k} \tau_j}
 \ge \frac{k+1}{k+2}.
\]
Roughly speaking, this means that, on $\Omega_0$, for every $\varepsilon>0$ infinitely often the last generation represents
at least a fraction $1-\varepsilon$ of the entire population. Whence
due to our choice of $\{f_{n,i}\}_{i \in \N}$, on $\Omega_0$, for every Borel set $A \subseteq [0,1]$ such that $\mu(A)>0$
and for every $\varepsilon>0$, we have $Z_n(A)/Z_n \ge 1-\varepsilon$ infinitely often.

\section{Proofs}
\label{sec:proofs}

We note that, for any fixed left interval $I$,  $\{Z_{2n}(I)\}_{n \in \N}$ is a random walk on $\N$ (with increments
depending on the position).
More precisely it is the queuing process 
associated to the i.i.d.~increments $\{\widetilde X_n-Y_n\}_{n \in \N}$, as defined 
by equation~\eqref{eq:XYtilde}; indeed
$$Z_{2n+2}(I)-Z_{2n}(I)=\max(-Z_{2n}(I),\widetilde X_{n+1}- Y_{n+1}).$$ 
We denote by $\{S_{n}(I)\}_{n \in \N}$ the random walk with independent increments,
where $S_n(I):=\sum_{i=1}^n (\widetilde X_i- Y_i)$.
The drift of this random walk is $\E[\widetilde X_i-  Y_i]=\E[\mu(I) X_i-Y_i]$ which is independent of
$i$. 

At time $0$ we have $S_0=Z_0(I)=0$
and for all $n$
\begin{equation}\label{eq:duality}
 Z_{2n}(I)=S_n(I)-\min_{i \le n} S_i(I)=\max_{i \le n} \sum_{k=i+1}^n (\widetilde X_k - Y_k), \quad \forall n \in \N.
\end{equation}

\noindent
By the Duality Principle, $Z_{2n}(I)$ and $\max_{k \le n} S_k(I)$ have the same law. Since $S_0(I)=0$ then 
$\min_{i \le n} S_i(I) \le 0$, hence $ Z_{2n}(I) \ge S_n(I)$ for all $n \in \N$.

Define $d:=
\mathrm{GCD}(n \in \Z \colon \pr(\widetilde X-Y=n)>0)$; 
by elementary number theory
it is easy to show that, since $\E[X]\E[Y]>0$, when $\mu(I)>0$ the random walk $\{Z_{2n}(I)\}_{n \in \N}$ (resp.~$\{S_{n}(I)\}_{n \in \N}$)
is irreducible on the set $\{d n \colon n \in \N\}$ (resp.~$\{d n \colon n \in \Z\}$). 

We start with the classification of the random walk $\{Z_{2n}(I)\}_{n \in \N}$. 
\begin{pro}[\textbf{Recurrence and transience}]\label{pro:typesofrecurrence} 
 Let $I$ be a left interval such that $\pr(\widetilde X \not = Y)>0$.  Denote by $\tau_n$ the time spent at $0$ by the
 random walk $\{Z_{2i}(I)\}_{i \in \N}$ up to time $n$.
 The random walk 
 is 
 \begin{enumerate}
  \item transient if and only if $\E[\mu(I)X-Y]\in (0,+\infty]$, 
  in this case $\pr(\sup_{n \in \N} \tau_n< +\infty)=1$; 
  \item null recurrent if and only if $\E[\mu(I)X-Y]=0$, 
  in this case $\pr(\lim_{n \to +\infty} \tau_n/n = 0)=1$;
  \item positive recurrent if and only if $\E[\mu(I)X-Y]\in[-\infty,0)$, 
  in this case $\pr(\lim_{n \to +\infty} \tau_n/n > 0)=1$. 
  Moreover, as $n \to +\infty$  
$$Z_{2n}(I)
  \stackrel{dist}{\to} S_\infty(I),\text{  a.s., where  }S_\infty(I):=\sup_{n\geq 0} S_n(I)<\infty, \quad{\rm a.s.}.$$
 \end{enumerate}
\end{pro}
 Note that the case where $\pr(\widetilde X = Y)=1$ is trivial, since it means that $\mu(I)=1$ and $X=Y=c \in (0,+\infty)$ a.s.;
 thus, $Z_{n}(I)$ equals $c$ when $n$ is odd and $0$ when $n$ is even.

\begin{proof}
Recall the relation between the random walks $\{S_n(I)\}_{n \in \N}$ and $\{Z_{2n}(I)\}_{n \in \N}$ given by equation~\eqref{eq:duality}. In particular
the return times to $0$ of the second process are the weak descending ladder times of the first one, that is, the times $n$ such that $S_n(I) \le S_i(I)$ 
for all $i \le n$. 
We denote by $\{T_i\}_{i \in \N}$ the sequence of intervals between two consecutive weak descending ladder times of $\{S_n(I)\}_{n \in \N}$
(that is, the times between two consecutive returns at $0$ of $\{Z_{2n}(I)\}_{n \in \N}$).
Note that $\{T_i\}_{i
\in \N}$ are i.i.d random variables and $T_1=\min\{n\geq 1: \,\,
S_n\leq 0\}$.

\begin{enumerate}
  \item 
  If $\E[\mu(I) X-Y] >0$ (either finite or
  infinite) then, by the SLLN, $S_n(I) \to +\infty$ a.s.,
  hence the same happens to the process $\{Z_{2n}(I)\}_{n \in
  \N}$ since $Z_{2n}(I) \ge S_n(I)$ for all $n \in\N$ (see
  equation~\eqref{eq:duality} and the remark afterwards).
 This implies that $\inf_{n \ge 0}
  S_n(I)=:S_{-\infty}>-\infty$ a.s. and the Markov chain
  $\{Z_{2n}(I)\}_{n \in \N}$  is transient. As a consequence
 $\pr(\sup_{n \in \N} \tau_n< +\infty)=1$.
    \item 
  When the distribution of $\widetilde X_n -Y_n$ is not degenerate (that is, it is not $\delta_0$) then according to 
    \cite[Theorem 4, Ch.VI.10]{cf:Feller2} $\{S_{n}(I)\}_{n \in \N}$ is a recurrent random walk on the set 
  $\{d n \colon n \in \Z\}$. 
    Since there are infinitely many reachable states on the left (as well as on the right) of the origin, 
     we have
  $\pr(T_1<\infty)=1$ so that $\{Z_{2n}(I)\}_{n \in \N}$ is
  recurrent.
  Moreover $\E[T_1]=+\infty$
  (see \cite[Theorem 2(i), Ch.XII.2]{cf:Feller2})   and this implies the null recurrence
   of $\{Z_{2n}(I)\}_{n \in \N}$. It is well known  that, for a
   recurrent random walk $\pr(\lim_{n \to +\infty} \tau_n/n = 1/\E[T_1])=1$ where, in this case, $1/\E[T_1]=0$.
    \item   
    We apply again the SLLN to $\{S_n(I)\}_{n \in \N}$ to deduce that $S_n(I) \to -\infty$ a.s., hence 
  $\sup_{n \ge 0}S_n(I) =:S_\infty(I)<+\infty$ and
  $\E[T_1]$ is finite (see \cite[Theorem 2(ii), Ch.XII.2]{cf:Feller2}). Thus,
  $\{Z_{2n}(I)\}_{n \in \N}$ is positive recurrent. 
  As before $\pr(\lim_{n \to +\infty} \tau_n/n = 1/\E[T_1])=1$ where, in this case, $1/\E[T_1]>0$.
  
  It is clear that $\max_{i \le n} S_n(I) \uparrow S_\infty(I)$ a.s.~and the conclusion follows by equation~\eqref{eq:duality}
  (see also \cite[Ch.VI.9]{cf:Feller2}).
 \end{enumerate}
\end{proof}

The next proposition deals with the a.s.~convergence of $Z_n(I)/n$ as $n\to \infty$.
\begin{pro}[\textbf{Law of large numbers}]\label{pro:slln}
\begin{enumerate}
  \item For every interval $I \subseteq [0,1]$,
 \begin{equation}\label{eq:shape}
   Z_{{n}}(I)/n \to {\frac{1}{2} \Big(} \mu(I \cap (f_c,1])\E[X] +\ident_I(f_c) \E[\mu([0,f_c])X-Y]\Big),\quad{\rm a.s.}
  \end{equation}
 \item If $\E[X-Y] \in [-\infty,0]$ then, for all sets $A \subseteq [0,1]$, $Z_n(A)/n \to 0$ almost surely as $n \to +\infty$.
\item
 Let $I$ be a left interval and $J \subseteq [0,1]$ be such that $I\cap J=\emptyset$ and $\mu(J)>0$.
 Suppose that $\E[\mu(I)X-Y] \in (0,+\infty]$.
  Then, a.s.,
 $Z_n(J)$ is nondecreasing eventually as $n \to +\infty$ and $Z_n(J)/n\to \mu(J)\E[X] /2$.
\end{enumerate}
\end{pro}

\begin{proof}
\ \\
 \begin{enumerate}
  \item 
  \textbf{Left interval $I$}. For a left interval $I$ equation~\eqref{eq:shape} becomes
 \begin{equation}\label{eq:shapefundamental}
     \frac{Z_{n}(I)}{n} \to \frac{1}{2} \E[\mu(I) X-Y] \vee 0,\quad{\rm
     a.s.}
    \end{equation}
  Let 
  $\Delta:=\E[\mu(I)X-Y] \in [-\infty,+\infty]$;  
  by the SLLN we have that \textbf{(a)}  $S_n(I)/n \to \Delta$ a.s..
  We separate two cases.

 \noindent $\bullet$ \textbf{$\Delta \in [-\infty, 0]$}.
Since $\liminf_n S_n(I)=-\infty$, then \textbf{(b)} for every $n_0$ there is a larger weak descending ladder time, 
i.e.~$n\ge n_0$ such that $S_n(I) \le S_k(I)$ for all $k \le n$.
 Hence almost every trajectory satisfies both \textbf{(a)} and \textbf{(b)}; let us consider such a trajectory.
 
 When $0 \ge \Delta>-\infty$ then 
 for every $\varepsilon>0$ there exists $n_0$ such that for every $n \ge n_0$ we have $|S_n(I)/n-\Delta| < \varepsilon/2$.
 Consider a weak descending ladder time $n_1 \ge n_0$;
 it is clear that, for every $n \ge n_1$ then $\min_{k \le n}S_k(I)=S_{k_n}(I)$ for some $k_n$ such that $n \ge k_n \ge n_1 \ge n_0$.

When $\Delta = 0$, then for every $n \ge n_1$ we have, by equation~\eqref{eq:duality},
 \[
 \begin{split}
  \frac{|Z_{2n}(I)|}n &= \Big |\frac{S_n(I)}n - \frac{\min_{k \le n}S_k(I)}n \Big | =
  \Big |\frac{S_n(I)}n - \frac{S_{k_n}(I)}n \Big | 
  \le \frac{|S_n(I)|}n + \frac{|S_{k_n}(I)|}{k_n} \cdot \frac{k_n}n < \varepsilon
 \end{split}
 \]
since $n \ge n_0$ and $n \ge k_n \ge n_0$.\\
When $-\infty< \Delta < 0$, take
 $\varepsilon \in (0, -2 \Delta)$.
 For every $n \ge n_1$ we have
 \[
  \Delta +\varepsilon/2 \ge \frac{S_n(I)}n \ge \frac{S_{k_n}(I)}n  \ge \frac{S_{k_n}(I)}{k_n} \ge \Delta -\varepsilon/2
 \]
since $S_{k_n}(I) \le n (\Delta +\varepsilon)<0$ (for all $n>0$) and
$n \ge k_n \ge n_0$. From the above chain of inequalities we obtain
$|S_{k_n}(I)/n -\Delta| \le \varepsilon/2$. Using again
equation~\eqref{eq:duality} we have
 \[
 \begin{split}
  \frac{|Z_{2n}(I)|}n &
  \le \Big |\frac{S_n(I)}n -\Delta \Big | + \Big |\frac{S_{k_n}(I)}{n} -\Delta \Big | < \varepsilon.
 \end{split}
 \]
 If $\Delta=-\infty$, consider the
 process $\{\widehat Z_{2n}\}_{n \in \N}$ constructed by using
 $Y_k \wedge M$ instead of $Y_k$ in such a way that
 $\E[\mu(I)X-Y\wedge M] \in (-\infty,0)$. We have $0 \le
 Z_{2n}(I)/n \le \widehat Z_{2n}(I)/n \to 0$ almost surely as $n \to
 \infty$.

 We are left to prove
 that
\begin{equation}\label{ou1}
  Z_{2n+1}(I)/n  \to 0,\quad{\rm a.s.}
\end{equation} 
Remember that, for all $\varepsilon
 >0$, $\E[X]<\infty$ iff $\sum_{n \in \N} \pr(X >
 \varepsilon n)<\infty$; thus, by the Borel-Cantelli's Lemma,
  $\E[X]<\infty$ implies $\pr(\liminf_n \{X_n  \le \varepsilon
  n\})=1$.
Thus
\begin{equation}\label{toke}
 \frac{|Z_{2n+1}(I)-Z_n(I)|}{n} \le \frac{X_{n+1}}n \to 0,\quad {\rm a.s.}
\end{equation}
so that from ${Z_{2n}/n}\to 0$, a.s., the
convergence (\ref{ou1}) follows.

\noindent $\bullet$ \textbf{$\Delta \in (0, +\infty]$}. 
By the SLLN, $S_n(I) \to +\infty$ a.s. and   $\inf_{n \ge 0} S_n(I)=:S_{-\infty}>-\infty$, almost surely.
  By \textbf{(a)}, using equation~\eqref{eq:duality}, we have
  $Z_{2n}(I)=S_n(I)-S_{-\infty}$ eventually a.s., which implies
  $Z_{2n}(I)/n  \to \E[\mu(I)X-Y]$, almost surely.
  
As before we are left to show 
\begin{equation}\label{ou}
  Z_{2n+1}(I)/n  \to \E[\mu(I)X-Y],\quad{\rm a.s.}
\end{equation}
If  $\E[X]<\infty$ then we use (\ref{toke}) to obtain (\ref{ou}).

\noindent
If $\E[X]=\infty$, note that 
 $Z_{2n+1}(I) \ge Z_{2n}(I)$,
thus $Z_{2n+1}(I)/n\to +\infty$ almost surely.
\medskip

\noindent
\textbf{Generic interval $I$}. Consider the the two left
intervals $I_1:=\{x \in [0,1] \colon \exists y \in I,\, x \le y\}$
and $I_2:=\{x \in [0,1] \colon x < y, \, \forall y \in I\}$. Clearly
$I_2 \subseteq I_1$, $I_1 \setminus I_2=I$, whence
$Z_{n}(I)=Z_{n}(I_1)-Z_{n}(I_2)$.
For $I_1$ and $I_2$, the convergence in equation~\eqref{eq:shapefundamental} holds.

If $f_c\in I$, then $I_2\subseteq [0,f_c)$  so
that $E[\mu(I_2)X-Y]\leq 0$ and by the result for left intervals, $Z_n(I_1)/n\to 0$ almost surely. Therefore, 
$$\lim_n \frac{Z_n(I)}{n}=\lim_n \frac{Z_n(I_1)}{n}=\frac{E[\mu(I_1)X-Y]}{2}=\frac{E[\mu(I\cup
[0,f_c))X-Y]}{2},\quad {\rm a.s.}.$$ 

Suppose $f_c\not\in
I$. If $I\subseteq [0,f_c)$, then $E[\mu(I_1)X-Y]\leq 0$ and
 ${Z_n(I)/ n}\to 0$ almost surely. If,
$I\subset (f_c,1]$, then by equation~\eqref{eq:shapefundamental} we have
$$\frac{Z_n(I)}{n}\to \frac{E[\mu(I)X-Y]}{2}=\frac{E[\mu(I\cap (f_c,1])X-Y]}{2},\quad {\rm a.s.}$$ and the statement now follows.

\item In this case $f_c=+\infty$. Whence $Z_n/n \to 0$ a.s., thus the same holds for $Z_n(A)$ for every $A \subseteq [0,1]$.

\item
The result follows from the fact that
$\lim_{n \to +\infty} Z_n(I) =+\infty$ almost surely.
Since $Z_n(I)\to +\infty,$ a.s.~then no species with fitness from $[0,1]\setminus I$ are
removed, eventually. 
By the SLLN, the number of births in
 $J$ (up to time $n$) divided by $n$ goes to its expectation almost surely as $n
\to +\infty$ and this yields the claim.
\end{enumerate}

\end{proof}

\begin{proof}[Proof of Theorem~\ref{th:gen1}]
(1) Take $I=[0,f]$ and suppose 
that $Z_{2 n_0}(I)>0$.
Since $Z_{2n+1}(I) \ge Z_{2n}(I)$ for all $n$, in
order to check whether the process hits the origin or not, it is
  enough to consider the process
  $\{Z_{2n}(I)\}_{n \in \N}$.
By Proposition~\ref{pro:typesofrecurrence}(1) $\{Z_{2n}(I)\}_{n \in \N}$ is transient and
  $\pr(Z_{2n}({I})>0,\
  \forall n>n_0|Z_{2 n_0}=i)>0$ (and this probability does not depend on $n_0$), for all $i\ge1$ such that $\pr(Z_{2 n_0}=i)>0$. 
Therefore we have survival. Moreover $Z_{2n}(I)\to\infty$ a.s., 
thus from $Z_{2n+1}(I)\ge Z_{2n}(I)$ we have $Z_n(I)\to\infty$ almost surely.
  
Suppose now that there are species with fitness
$f$ alive at time $n_0$. If $Z_n(I)$ is never empty for $n \ge n_0$, then the
  fitness $f$ survives.  Thus the probability of
  survival of $f$ equals to the probability that $Z_n(I)$ is always
  positive. 

We now show that $Z_n(\{f\})\to\infty$.
If $Z_{2n}(\{f\})>Z_{2n+2}(\{f\})$ then
$$Z_{2n}({I}) \ge
Z_{2n}(\{f\})>Z_{2n+2}(\{f\})={Z_{2n+2}(I)}.$$
 For all $M>0$ let $n(M)$ be a random integer such that
  $Z_{2n}({I}) \ge M$ for all $n \ge n(M)$. For
  every $n \ge n(M)$ either $Z_{2n}(\{f\}) \le Z_{2n+2}(\{f\})$
  or $Z_{2n+2}(\{f\}) \ge M$. In particular,
  if $Z_{2n_1}(\{f\}) \ge M$ for some $n_1>n(M)$, 
  then it is true for  all $n \ge n_1$.
Consider the first (random) time $n_1 \ge n(M)$ when
  $Z_{2n_1}([0,f)]=0$. 
  If $n_1=\infty$ then, from time $n(M)$ on, $Z_n(\{f\})$ is non
  decreasing and strictly increasing infinitely many often; indeed,
  species of fitness $f$
  are created infinitely many times a.s.~(since $\mu(\{f\})>0$) and these species will
  be never removed after time $n(M)$ (since $Z_{2n}([0,f)]>0$ for every $n\geq n(M)$). 
  If, on the other hand, $n_1<\infty$ then $Z_{2n}(\{f\})
  \ge M$
   for all $n \ge n_1$ and the result follows.
  
  
(2)  By
  Proposition~\ref{pro:typesofrecurrence}, 
  applied to $I=[0,f]$, the process
  $\{Z_{2n}(I)\}_{n\in\mathbb{N}}$ is recurrent and so
  $Z_n(I)=0$ infinitely often, almost surely.
\end{proof}

 \begin{proof}[Proof of Corollary~\ref{rem:dynamics}]
The statement (1) follows from
the equation~\eqref{eq:killing}. For every left interval $I$ such that
$\E[\mu(I)X-Y] > 0$ by equation~\eqref{eq:shapefundamental}
we have $Z_n(I) \to +\infty$ a.s.~and there are no more particles
killed in $I^c$ eventually as $n \to +\infty$. This
implies the first statement of (3) and the statement (5).
Conversely, if $\E[\mu(I)X-Y] \le 0$ then by
equation~\eqref{eq:killing} we have $K_{n}(I^c)\sim
-\frac12\E[\mu(I)X-Y]$ almost surely as $n \to +\infty$.
This implies (4). Finally, if $\E[\mu(I)X-Y] = 0$
then $Z_n(I)=0$ i.o.~almost surely, whence by
$\E[X-Y]>0$ it follows that $K_n(I^c)\to +\infty$ as $n \to
+\infty$ almost surely. This implies the second
statement of (3) (applied to the case $F(f_c)\E[X]-\E[Y]=0$) and the
statement (2) (applied to the case $F(f_c-)\E[X]-\E[Y]=0$).
\end{proof}


\begin{proof}[Proof of Theorem~\ref{th:genshape1}]
\begin{enumerate}
   \item  It is enough to consider $A \subseteq [0,1] \cap [0,f_c)$ since there are no births in $A \setminus ([0,1] \cap [0,f_c))$ almost
  surely.
  It follows immediately from  Proposition~\ref{pro:slln}(1); the uniform convergence
  comes from the inequality $Z_n(A) \le Z_n(I)$ for all $A \subseteq I$ and $n \in \N$.
  \item
  By Proposition~\ref{pro:slln}(1) we have that
$Z_{n}/n \to +\infty$ almost surely.
    Recall that $f_{n,1},f_{n,2},\ldots$ are
  i.i.d.
  We start by noting that
  \[
   \sum_{i=1}^{\lfloor(n+1)/2\rfloor}(\widetilde X_i-Y_i)\le Z_n(A)\le \sum_{i=1}^{\lfloor(n+1)/2\rfloor}\widetilde X_i,
  \]
  whence
  \[
   \frac{ \sum_{i=1}^{\lfloor(n+1)/2\rfloor}(\widetilde X_i-Y_i)}{ \sum_{i=1}^{\lfloor(n+1)/2\rfloor} X_i}\le \frac{Z_n(A)}{Z_n}\le 
   \frac{ \sum_{i=1}^{\lfloor(n+1)/2\rfloor}\widetilde X_i}{ \sum_{i=1}^{\lfloor(n+1)/2\rfloor} (X_i-Y_i)}.
  \]
Moreover, since a.s.
\[
 \frac{\sum_{i=1}^kY_i}{\sum_{i=1}^kX_i}\stackrel{k\to\infty}{\to}0,
\]
it suffices to prove that a.s.
\[
 \frac{\sum_{i=1}^k\widetilde X_i}{\sum_{i=1}^kX_i}\stackrel{k\to\infty}{\to}\mu(A).
 \]
We recall that $\widetilde X_i$ is a sum of $X_i$ Bernoulli random variables of parameter $\mu(A)$
and that the family of these Bernoulli variables is independent of the family  $\{X_n\}_{n\in\N}$.
Thus 
\[
 \frac{\sum_{i=1}^k\widetilde X_i}{\sum_{i=1}^kX_i}=\frac{\mathcal B(\sum_{i=i}^kX_i,\mu(A))}{\sum_{i=1}^kX_i},
 \]
and we are therefore left to prove that a.s.
\begin{equation}\label{eq:bernoulli}
 \frac{\sum_{i=1}^{N_n}W_i}{N_n}\stackrel{n\to\infty}{\to}\mu(A),
 \end{equation}
whenever $\{N_n\}_{n\in\N}$ is a sequence of random variables
such that $N_n\to\infty$ a.s., $\{W_n\}_{n\in\N}$ is a sequence of independent Bernoulli random variables 
of parameter $\mu(A)$ and the two sequences are independent.

Define the sequence of stopping times $\tau_n:=\inf\{k \colon N_k \ge n\}$; clearly equation~\eqref{eq:bernoulli} holds
if and only if $\frac{\sum_{i=1}^{N_{\tau_n}}W_i}{N_{\tau_n}}\stackrel{n\to\infty}{\to}\mu(A)$.
By the Hoeffding's inequality
$$
\pr \Big(\big|\frac{\sum_{i=1}^{N_{\tau_n}}W_i}{N_{\tau_n}}-\mu(A)\big|\geq
\varepsilon \Big|N_{\tau_n}\Big)\leq 2\exp[-2\epsilon^2 N_{\tau_n}]\leq
2\exp[-2\epsilon^2 n],$$ 
so that after taking expectation (over $N_n$)
$$
\pr \Big(\big|\frac{\sum_{i=1}^{N_{\tau_n}}W_i}{N_{\tau_n}}-\mu(A)\big|\geq
\varepsilon \Big)\leq  2\exp[-2\epsilon^2 n],
$$ and the statement
follows from the Borel-Cantelli's Lemma.

%

 \item
 The a.s.~convergence $Z_n/n \to \E[X-Y]/2$ as $n \to +\infty$ comes from Proposition~\ref{pro:slln}.
 
  As for the second part, if $A$ is an interval then the claim follows trivially by applying
 Proposition~\ref{pro:slln} to $Z_{n}(I)$ and
 $Z_{n}$. 
  Let
$${\cB}:= \{\cup_{i=1}^k [a_i,b_i]: a_1< b_1<a_2<b_2<\cdots <b_k,
a_i,b_i\in\mathbb{Q},\quad k=1,2,\ldots \}.$$ Since there are
countable many intervals with rational endpoints,
\begin{equation}\label{borel2} \pr\Big( \frac{Z_n(B)}{Z_n}\to P(B),\quad \forall
B\in {\cB} \Big)=1.
\end{equation}
By the regularity of probability measures, it is easy to see that
for every Borel set $A$, for every $\varepsilon >0$ there exists
sets $B_1, B_2\in {\cB}$, both depending on $\varepsilon$ such that
$B_1\subset A \subset B_2$ and $P(B_2\setminus B_1)\leq
\varepsilon$. Thus, if $Z_n(B)/Z_n\to P(B)$ for every $B\in \cB$,
then for every $\varepsilon >0$
\[
\begin{split}
\limsup_n \frac{Z_n(A)}{Z_n} &\leq \limsup_n \frac{Z_n(B_2)}{Z_n}=P(B_2)\leq P(A)+\varepsilon,\\
\liminf_n \frac{Z_n(A)}{Z_n} 
&\geq \liminf_n \frac{Z_n(B_1)}{Z_n}=P(B_1)\geq
P(A)-\varepsilon,
\end{split}
\] so that equation~\eqref{borel2} implies
equation~\eqref{eq:shape1}.

By the above arguments, 
it suffices to show the
following: if $P_n$, $P$ are probability measures on $\mathbb{R}$,
so that for every Borel set $A$, $P_n(A)\to P(A)$, then
$\sup_t|F_n(t)-F(t)|=:\|F_n-F\|_\infty\to 0$, where $F_n$ and $F$ are the
corresponding distribution functions. Let $\{x_i\}$ be the set of
atoms on $P$, let $\{p_i\}$ be their masses and let
$H(t)=\sum_{i=1}^{\infty}p_i I_{(-\infty,t]}(x_i)$ be the
distribution function of the subprobability measure
$\sum_{i=1}^{\infty}p_i \delta_{x_i}$. Let, for every $i$,
$p_i^n=P_n(\{x_i\})$, by assumption $p_i^n\to p_i$. Let $H_n$ be the
distribution function of the subprobability measure
$\sum_{i=1}^{\infty}p^n_i \delta_{x_i}$. Since
$P_n(\{x_1,x_2,\ldots\})\to P(\{x_1,x_2,\ldots\})$, from Scheffe's
theorem, it follows that $\|H_n-H\|_\infty\to 0$. Since for every $t$,
$F_n(t)\to F(t)$, we have that $(F_n(t)-H_n(t))\to (F(t)-H(t))$ for
every $t$. The function  $F(t)-H(t)$ is a continuous distribution
function of a subprobability measure, so the pointwise convergence
implies uniform convergence.  So, $\|F_n-F\|_\infty\leq
\|H_n-H\|_\infty+\|(F_n-H_n)-(F-H)\|_\infty\to 0.$
 
\end{enumerate}
\end{proof}

The following Lemma is well known and we include it for the sake of completeness.

\begin{lem}\label{lem:genfunct}
Consider two $\N$-valued random variables $X$ and $Y$ on $\N$ with laws $\rho_X$ and $\rho_Y$ respectively; let $\Phi_X(z)=\E[z^X]=\sum_{i=0}\rho_X(i)z^i$ and 
 $\Phi_Y(z)=\E[z^Y]=\sum_{i=0}\rho_Y(i)z^i$ the corresponding generating functions. Let $\{X_i\}_{i \in \N}$ be a i.i.d.~sequence of random  variables with law $\rho_X$.
 Finally let $\{Z_i\}_{i \in \N}$ be a generic sequence of $\N$-valued random variables with laws $\{\rho_{Z_n}\}_{n \in \N}$
 and generating functions $\{\Phi_{Z_n}\}_{n \in \N}$.
 \begin{enumerate}
  \item If $Z=\sum_{i \in \N} \ident_{\{X=i\}} Z_i$ then its law is $\rho_Z=\sum_{i \in \N} \rho_X(i) \rho_{Z_i}$ and the generating function is 
  $\Phi_Z=\sum_{i \in \N} \rho_X(i) \Phi_{Z_i}$.
    \item If $Z=\sum_{i=1}^n Z_i$, where $\{Z_1, \ldots, Z_n\}$ are independent, then $\rho_Z=\rho_{Z_1}\ast\cdots\ast \rho_{Z_n}$ (where $\ast$ denotes
    the usual convolution) and  $\Phi_Z=\prod_{i=1}^n \Phi_{Z_i}$.
  \item If $Z=\sum_{i =0}^Y X_i$ then the law of $Z$ is $\rho_Z= \sum_{i \in \N} \rho_Y(i) (\ast^i \rho_{X})$ (where $\ast^i \rho_{X}$ is the convolution
  of $i$ copies of $\rho_{X}$) and $\Phi_Z=\Phi_Y \circ \Phi_X$.
 \end{enumerate}
\end{lem}

\begin{proof}
 \begin{enumerate}
  \item It is straightforward.
  \item The explicit expression of the law is trivial and $\rho_Z=\E[z^{\sum_{i=1}^n Z_i}]=\E[{\prod_{i=1}^n z^{Z_i}}]=\prod_{i=1}^n \E[{z^{Z_i}}]$ where the last
  equality comes from the independence.
  \item The explicit expression of the law follows by conditioning $Z$ on $Y$. Then (1) and (2) yield the conclusion.
  \end{enumerate}

\end{proof}

\begin{proof}[Proof of Proposition~\ref{pro:BPcase}]
Suppose that $Z_{2n+k}([0,f])=0$ for some $k \ge 0$ then  $k_0:=\min\{k\colon Z_{2n+k}([0,f])=0\}<+\infty$ and $k_0$ must be even.
Hence there are no species in $[0,f]$ at some even time larger than $2n+1$ if and only if there is just one species at some odd time in $[0,f]$ (larger
than $2n+1$).
The probability of having just one species at some odd time (larger than $2n+1$) provided there are $i$ species at time $2n+1$ equals
the probability of having no species  at some even time (larger than $2n+1$) provided there are $i-1$ species at time $2n+1$; thus
the case $j=2$ follows from the case $j=1$.

 Let us take $j=2$ an consider the process $\{Z_{2n+1}([0,f])\}_{n \in \N}$.
 Until all species in $[0,f]$ are gone, each time a species is removed it is replaced by a random number of species (in $[0,f]$) 
 with generating function $\Psi_a$. This is equivalent to a branching process with generating function $\Psi_a$. 
 Equation~\eqref{eq:BP1} follows from standard results in Branching Process theory.
The equivalence $\bar q_f <1 \Longleftrightarrow F(f) \le 1/\Phi^\prime(1)$ follows from the assumption $\pr(X_k=1)<1$ and 
from the equality $\Psi_f^\prime(1)=F(a)\Phi^\prime(1)$.

It is straightforward to prove that there is almost sure (temporary) extinction in $[0,f_c]$ if and only if $F(f_c) \le 1/\Phi^\prime(1)$.

The case of the process $\{Z_n([0,f))\}_{n \in \N}$ is completely analogous. In this case we just need to note that 
$F(f_c^-)\Phi^\prime(1) \le 1$.
\end{proof}


\section*{Acknowledgements}
The authors are grateful to Mauro Ghidelli for carefully reading the manuscript and for useful
remarks.
The first and last authors acknowledge support by INdAM and Prin 2015.
The second author is supported by the Estonian institutional
research funding IUT34-5.

\end{document}